\documentclass[11pt]{amsart}
\numberwithin{equation}{section}
\usepackage[english]{babel}
\usepackage[T1]{fontenc}
\usepackage[latin1]{inputenc}
\usepackage{indentfirst}
\usepackage{enumitem}
\usepackage{amsmath,amssymb, amsbsy}
\usepackage{comment}
\usepackage{amsfonts}
\usepackage{hyperref}
\usepackage{cleveref}
\usepackage{esint}
\usepackage{latexsym}
\usepackage{amsthm}
\usepackage[dvips]{graphicx}
\usepackage{xcolor}
\usepackage{tikz}
\usepackage{pgfplots}
\usetikzlibrary{intersections}
\usepackage{tikz-3dplot}
\usepackage[outline]{contour}
\DeclareGraphicsExtensions{.pdf,.png,.jpg,.eps}

\usepackage{bbm}

\usepackage[font=small,labelfont=bf]{caption}

\usepackage{amsthm}

\usepackage{hyperref}

\usepackage{xcolor}
\usepackage[latin1]{inputenc}
\usepackage[active]{srcltx}
\definecolor{Arancio}{cmyk}{0,0.61,0.87,0}

\definecolor{blus}{RGB}{0,102,204}

\usepackage{doi}

\let\arXiv\arxiv

\newcommand{\brd}[1]{\mathbb{#1}}
\newcommand{\R}{\brd{R}}

\newcommand{\N}{\brd{N}}

\newcommand{\be}{\begin{equation}}
\newcommand{\ee}{\end{equation}}

\newtheorem{teo}{Theorem}[section]
\newtheorem{Corollary}[teo]{Corollary}
\newtheorem{Lemma}[teo]{Lemma}
\newtheorem{Theorem}[teo]{Theorem}
\newtheorem{Proposition}[teo]{Proposition}
\theoremstyle{definition}
\newtheorem{Definition}[teo]{Definition}

\newtheorem{remark}[teo]{Remark}

\newcommand{\D}{\nabla}
\newcommand{\dive}{\operatorname{div}}

\usepackage{amsmath}
\usepackage{tikz}
\usepackage{mathdots}
\usepackage{yhmath}
\usepackage{cancel}
\usepackage{color}
\usepackage{siunitx}
\usepackage{array}
\usepackage{multirow}
\usepackage{amssymb}
\usepackage{textcomp}
\usepackage{gensymb}
\usepackage{tabularx}
\usepackage{extarrows}
\usepackage{booktabs}
\usetikzlibrary{fadings}
\usetikzlibrary{patterns}
\usetikzlibrary{shadows.blur}
\usetikzlibrary{shapes}
\pgfplotsset{compat=1.18} 
\begin{document}


\title[Elliptic equations degenerating on lower dimensional manifolds]
{Remarks on elliptic equations degenerating on lower dimensional manifolds}

\author{Gabriele Cora, Gabriele Fioravanti, Stefano Vita}

\address{Gabriele Cora, D\'epartement de Math\'ematique, Universit\'e Libre de Bruxelles, Boulevard du Triomphe 155, 1050, Brussels, Belgium}
\email{gabriele.cora@ulb.be}

\address{Gabriele Fioravanti, Dipartimento di Matematica "G. Peano", Universit\`a degli Studi di Torino, Via Carlo Alberto 10, 10124, Torino, Italy}
\email{gabriele.fioravanti@unito.it}

\address{Stefano Vita, Dipartimento di Matematica "F. Casorati", Universit\`a di Pavia, Via Ferrata 5, 27100, Pavia, Italy}
\email{stefano.vita@unipv.it}

\begin{abstract}
The paper continues the analysis started in \cite{CorFioVit25,Fio24} on the local regularity theory for elliptic equations having coefficients which are degenerate or singular on some lower dimensional manifold. The model operator is given by $L_au(z)=\mathrm{div}(|y|^a\nabla u)(z)$, where $z=(x,y)\in\mathbb R^{d-n}\times\mathbb R^n$, $2\leq n\leq d$ are two integers and $a\in\mathbb R$. The weight term is degenerate/singular on the (possibly very) thin characteristic manifold $\Sigma_0=\{|y|=0\}$ of dimension $0\leq d-n\leq d-2$. Whenever $a+n>0$, we prove smoothness of the axially symmetric $L_a$-harmonic functions. In the mid-range $a+n\in(0,2)$, we deal with regularity estimates for solutions with inhomogeneous conormal boundary conditions prescribed at $\Sigma_0$, and we establish the connection with fractional Laplacians on very thin flat manifolds via Dirichlet-to-Neumann maps, as a higher codimensional analogue of the extension theory developed by Caffarelli and Silvestre. Finally, whenever $a+n<2$ we complement the study in \cite{Fio24}, providing some regularity estimates for solutions having a homogeneous Dirichlet boundary condition prescribed at $\Sigma_0$ by a boundary Harnack type principle.
\end{abstract}

\maketitle

\section{Introduction}
In this paper, continuing the analysis in \cite{CorFioVit25,Fio24}, we investigate some further regularity features of the weak solutions to second order elliptic equations in divergence form which are uniformly elliptic
far from a characteristic manifold of low dimension. The model operator is given by
\begin{equation*}\label{generalPDE}
L_au(z)=\mathrm{div}(|y|^a\nabla u)(z).
\end{equation*}
Here $z=(x,y)\in\R^{d-n}\times\R^n$, $2\leq n\leq d$ are two integers, and $a\in\R$. The weight term is a power of the distance to the characteristic flat manifold of low dimension $0\leq d-n\leq d-2$
$$\Sigma_0=\{z=(x,y)\in\R^d \, \mid \, |y|=0\};$$
that is, $|y|=\mathrm{dist}(z,\Sigma_0)$.

\subsection*{Smoothness of axially symmetric solutions}
The local regularity of weak solutions of elliptic equations involving $L_a$ is strongly influenced by the presence of the weight term. Even in the case of $L_a$-harmonic functions, there is a natural threshold for regularity given by the combined parameters power-codimenson $(a,n)$. This can be seen at the level of Liouville type theorems, which classify the entire solutions having some polynomial growth, see \cite[Theorem 1.4]{CorFioVit25} and \cite[Theorem 1.4]{Fio24}. However, when the weight is locally integrable, i.e. $a+n>0$, $L_a$-harmonic functions with a homogeneous conormal boundary condition at $\Sigma_0$ and which additionally exhibit a particular symmetry, are actually smooth.

The class of weak solutions we are referring to are elements of the weighed Sobolev space $H^{1,a}(B_1):=H^1(B_1,|y|^adz)$ which satisfy
\begin{equation}\label{weakconormal}
\int_{B_1}|y|^a\nabla u\cdot\nabla\phi=0\qquad\mathrm{for \ every \ }\phi\in C^{\infty}_c(B_1).
\end{equation}
We refer to \cite{CorFioVit25} for a complete treatment of the functional framework of the problem. The equation is satisfied \emph{across} the thin manifold $\Sigma_0$, and this implies a formal homogeneous conormal condition at $\Sigma_0$
\begin{equation}\label{conormalweighted}
\lim_{|y|\to0}|y|^{a+n-1}\nabla u\cdot\frac{y}{|y|}=0,
\end{equation}
see Remark \ref{rem:conormal} for the precise meaning of the above expression. %Notice that \eqref{conormalweighted} corresponds to a vanishing weighted radial derivative with respect to the degenerate variables.

Then, the additional geometric condition which provides smoothness is the axial symmetry with respect to $\Sigma_0$; that is, solutions which are radial-in-$y$
\begin{equation*}
    u(x,y)=u(x,|y|).
\end{equation*}
The latter is proved in Theorem \ref{axiallysmooth}.

\subsection*{The inhomogeneous conormal problem}

The local weighted $H^{1,a}$-capacity of the thin manifold $\Sigma_0$ is key to understanding what kind of solutions one may face. When $a+n\geq2$ (superdegenerate setting) the weighted capacity is zero, and the homogeneous conormal condition in \eqref{conormalweighted} is naturally enjoyed by any weak solution. On the contrary, when $a+n\leq0$ (supersingular setting) the weighted capacity is infinite and solutions must satisfy a homogeneous Dirichlet condition at $\Sigma_0$. In the mid-range $a+n\in(0,2)$, the positive (and locally finite) weighted capacity of the thin set $\Sigma_0$ enables the imposition of inhomogeneous Dirichlet and conormal boundary conditions. Here we deal with the latter; that is, solutions to
 \begin{equation*}
    \begin{cases}
        -\dive(|y|^a \D u) = 0, & \text{ in } B_1,\\
        \displaystyle-\lim_{|y|\to 0}
        |y|^{a+n-1} \D u\cdot \frac{y}{|y|} = g(x), & \text{ on }\Sigma_0 \cap B_1,
    \end{cases}
    \end{equation*}
which are well defined (see Definition \ref{Def:inhomogeneous.conormal}) as elements of $H^{1,a}(B_1)$ satisfying
    \[
    \int_{B_1}|y|^a \D u \cdot \D \phi dz = \omega_n\int_{\Sigma_0 \cap B_1} g(x) \phi(x,0) dx, \qquad \text{for every }\phi\in C_c^\infty(B_1),
    \]
where $\omega_n = |\mathbb{S}^{n-1}|$. Then, we prove some regularity results for solutions of the above problem, as a corollary of our main theorems in \cite{CorFioVit25}, together with the regularity of axially symmetric solutions, see Proposition \ref{Prop:inhomogeneous}.

\subsection*{Higher codimensional extensions of fractional Laplacians} 
Motivated by \cite{CafSil07}, one might wonder whether it is possible to extend functions $u$ defined in $\R^{d-n}$ having a well-defined $s$-fractional Laplacian ($s\in(0,1)$) to the whole of $\R^d$, introducing $n$ additional variables.
Under the assumption $d-n>2s$, which allows for the definition of suitable energy spaces, in Section \ref{sec:10} we show that such an extension is given by the convolution $u * P$ with the Poisson-type kernel 
\[P(x,y)=P(|x|,|y|)=\frac{\Gamma(\frac{d-n + 2s}{2})}{\pi^{\frac{d-n}{2}}\Gamma(s)}\frac{|y|^{2s}}{(|x|^2+|y|^2)^{\frac{d-n+2s}{2}}},\qquad (x,y)\in\R^{d-n}\times\R^n.
\]
The extension is a radial-in-$y$ solution (axisymmetric with respect to $\Sigma_0=\R^{d-n}$) to 
$$\begin{cases}
-\mathrm{div}(|y|^a\nabla u)=0 &\mathrm{in \ }\R^d\setminus\Sigma_0\\
\displaystyle-\lim_{|y|\to0}|y|^{a+n-1}\nabla u\cdot\frac{y}{|y|}=d_{a,n}(-\Delta)^su(x,0) &\mathrm{on \ } \Sigma_0,
\end{cases}$$
where $a+n=2-2s\in(0,2)$, and $d_{a,n}$ is an explicit positive constant. As can be seen through the change of variables $|y| = r$, we exactly recover the classical Caffarelli-Silvestre extension. It would be interesting to study conformal fractional Laplacians and their higher codimensional extension theories in the setting of curved thin manifolds, following the approach of \cite{ChaGon11}.

\subsection*{The homogeneous Dirichlet problem}
Whenever $a+n<2$, one can provide regularity results for solutions to the homogeneous Dirichlet problem
\begin{equation}\label{eq:dirichlet}
\begin{cases}
-\mathrm{div}(|y|^a\nabla u)=0 &\mathrm{in \ } B_1\setminus\Sigma_0\\
u=0 &\mathrm{on \ } B_1\cap\Sigma_0.
\end{cases}
\end{equation}
The problem above has connections with the work of Nguyen on harmonic maps with prescribed singularities in general relativity \cite{Ngu11}. Moreover, it has also profound connections with the works of David-Feneuil-Mayboroda on the Dirichlet problem on lower dimensional boundaries \cite{DavFenMay19,DavFenMay21}, see also \cite{DaiFenMay23,DavMay22}.

In the same spirit of \cite{SirTerVit21b}, as $a+n<2$ one can provide regularity results for solutions $u$ to \eqref{eq:dirichlet} by considering the ratio $w=u/u_0$ between $u$ and the characteristic Dirichlet solution
\[
u_0(y)=|y|^{2-a-n},
\]
which solves again \eqref{eq:dirichlet}. The ratio $w$ is a solution of a degenerate problem $-\mathrm{div}(|y|^b\nabla w)=0$ with exponent $b=4-a-2n$ which lies in the superdegenerate range since $b+n>2$. Hence, $C^{0,\alpha}$ and $C^{1,\alpha}$ estimates for $w$ are implied by the main theorems in \cite{CorFioVit25}, see Corollary \ref{C:Dirichlet}. This improves some results in \cite{Ngu11}, and also slightly improves some regularity results in \cite{Fio24}, ensuring the sharp $C^{2-a-n}$ regularity of solutions under additional requirements on the codimension $n$. For instance, in the case $a+n=1$ and whenever $n\geq 4$, solutions to \eqref{eq:dirichlet} are provided to be Lipschitz continuous, see Remark \ref{rem:DFM}.

\subsection*{Notation}
We establish the notation that will be used throughout the paper.

\begin{itemize}[left=0pt]

    \item[$\cdot$] Let $2\leq n\leq d$ be two integers and consider the coordinates $z=(x,y)\in\R^{d-n}\times\R^n$.

    \item [$\cdot$] For $m \in \N$ we denote by $B_R^{m}(\zeta)$ the open $m$-dimensional ball of radius $R>0$ centered at $\zeta \in \R^m$. To ease the notation, we simply write $B_R^m= B^m_R(0)$ when $\zeta = 0$, and $B_R(\zeta)= B^d_R(\zeta)$, when the dimension $m=d$. In particular, $B_R = B^d_R(0)$.

%\item [$\cdot$]  Throughout the paper, any positive constant whose value is not important is denoted by $c$. It may take different values at different places.

\end{itemize}

\section{Smoothness of axially symmetric solutions}

In this section, we prove that axially symmetric - with respect to $\Sigma_0$ - solutions (radial-in-$y$) are locally smooth, by establishing a one-to-one correspondence with solutions to an equation that degenerates on a hyperplane.

\begin{Definition}\label{def:weak:sol:hom:0}
Let $a+n>0$. We say that $u$ is $L_a$-harmonic in $B_1$ across $\Sigma_0$, i.e. is a weak solution to
\begin{equation}\label{eq:weak:sol:conormal:0}
-\dive(|y|^a \nabla u)= 0, \quad \text{in }B_1
\end{equation}
if $u\in H^{1,a}(B_1)$ and satisfies
\begin{equation}\label{eq:weak:sol:integral}
\int_{B_1}|y|^a A\D u\cdot \D\phi = 0,
\end{equation}
for every $\phi\in C_c^\infty(B_1).$
\end{Definition}

\begin{remark}\label{rem:conormal}
  We highlight that our notion of weak solution implies a \emph{weighted conormal boundary condition} on the lower-dimensional set $\Sigma_0$. This is a consequence of the fact that the weak formulation \eqref{eq:weak:sol:integral} involves test functions whose support may touch the thin manifold $\Sigma_0$. We refer to these solutions as \emph{solutions across} $\Sigma_0$. Let us fix $\phi \in C_c^\infty(B_1)$. Multiplying \eqref{eq:weak:sol:conormal:0} by $\phi$ and integrating by parts in $B_1 \setminus \Sigma_\varepsilon$ (where $\Sigma_\varepsilon=\{|y|\leq\varepsilon\}$), we obtain
\[
0 = \int_{B_1 \setminus \Sigma_{\varepsilon}} -\dive(|y|^a\nabla u) \phi \, dz = \int_{B_1 \setminus \Sigma_{\varepsilon}} |y|^a\nabla u \cdot \nabla \phi \, dz - \int_{\partial \Sigma_\varepsilon \cap B_1} |y|^a \phi \nabla u \cdot \nu \, d\sigma.
\]  
Formally, taking the limit as $\varepsilon \downarrow 0$, we find
\[
\int_{B_1} |y|^a \nabla u \cdot \nabla \phi \, dz  = \int_{B_1^{d-n}} \mathcal{D}_u(x) \phi(x, 0) \, dx,
\]  
where
\[
\mathcal{D}_u(x) := -\lim_{\varepsilon \downarrow 0} \varepsilon^{1-n} \int_{\partial B_{\varepsilon}^n} |y|^{a+n-2} \nabla u \cdot y \, d\sigma(y).
\]  
Hence, in the weak formulation of \eqref{eq:weak:sol:conormal:0}, we are assuming that $\mathcal{D}_u = 0$. It is worth noting that if $a + n \geq 2$, due to the zero weighted capacity of the thin manifold, only solutions to \eqref{eq:weak:sol:conormal:0} make sense, since one could not impose any different boundary condition at $\Sigma_0$. Conversely, when $a + n \in (0, 2)$, the weighted capacity of $\Sigma_0$ is positive and locally finite, allowing the imposition of both inhomogeneous Dirichlet and inhomogeneous conormal boundary conditions, respectively $u=f$ and $\mathcal{D}_u = g $ on $\Sigma_0$.

\end{remark}

\begin{Definition}
    We say that a function $u:B_1\subset\R^d\to \R$ is axially symmetric in $y$ (i.e. with respect to $\Sigma_0$) if there exists a function $\tilde{u}:B_1^+\subset \R^{d-n+1}\to \R$ such that $u(x,y)=\tilde{u}(x,|y|)$, where $B_1^+ :=B_1^{d-n+1}\cap\{(x,r)\in \R^{d-n}\times \R \mid r>0\}$ denotes the unitary upper half ball in $\R^{d-n+1}$.
\end{Definition}

\begin{Lemma}\label{L:axially:1-1}
    Let $a+n>0$ and let $u$ be an axially symmetric in $y$ function.
    Then, $u$ is $L_a$-harmonic in $B_1$ across $\Sigma_0$ (weak solutions of \eqref{eq:weak:sol:conormal:0}) if and only if the function $\tilde{u}(x,r)$ is a weak solution to 
\begin{equation}\label{eq:symm2}
\begin{cases}
-\mathrm{div}(r^{a+n-1} \nabla\tilde{u})=0, &\mathrm{in \ } B_1^+,\\
\displaystyle\lim_{r\to 0}r^{a+n-1}\partial_r \tilde{u} = 0, &  \mathrm{on \ } B_1\cap\{r=0\},
\end{cases}
\end{equation}
in the sense that $\tilde u \in H^{1}(B_1^+,r^{a+n-1}dxdr)$ and 
\[
\int_{B_1^+}r^{a+n-1}\D \tilde{u} \cdot \D \phi\, dxdr =0,\qquad \text{for every } \phi \in C_c^\infty(B_1)\,.
\]
\end{Lemma}

\begin{proof}
First, we notice that, by taking the spherical-in-$y$ change of variable, one has
\begin{align*}
    \int_{B_1^+} r^{a+n-1} (\tilde{u}^2 + |\D \tilde{u}|^2)dxdr &= 
c_n\int_{\mathbb{S}^{n-1}} \int_{B_1^+} r^{a+n-1} (\tilde{u}^2 + |\D \tilde{u}|^2)dxdr d\sigma\\
&= c_n
\int_{B_1} |y|^{a} (u^2 + |\D u|^2)dxdy,
\end{align*}
where $c_n = |\mathbb{S}^{n-1}|^{-1}$. Then, $u \in H^{1,a}(B_1)$ if and only if $\tilde{u} \in  H^{1}(B_1^+,r^{a+n-1}dxdr)$.

Next, let us suppose that $u$ is a weak solution to \eqref{eq:weak:sol:conormal:0}. Fix $\tilde \phi=\tilde \phi(x,r) \in C_c^\infty(B_1^{d-n+1})$ and call $\phi(x,y) := \tilde \phi(x,|y|)$. Then,
\begin{align*}
    \int_{B_1^+} r^{a+n-1} \D \tilde u\cdot \D \tilde\phi\, dxdr &= c_n\int_{\mathbb{S}^{n-1}}\int_{B_1^+} r^{a+n-1} \D \tilde u\cdot \D \tilde\phi\, dxdr d\sigma\\
    &= c_n\int_{B_1}|y|^a \D u \cdot \D  \phi\, dxdy = 0\,.
\end{align*}
Hence, $\tilde{u}$ is a weak solution to \eqref{eq:symm2}.

Instead, let us suppose that $\tilde u$ is a weak solution to \eqref{eq:symm2}, fix $\phi=\phi(x,y) \in C_c^\infty(B_1)$ and define $\tilde{\phi}(x,r):= \int_{\mathbb{S}^{n-1}}\phi(x,r\sigma)d\sigma$. Then, one has
\begin{align*}
    \int_{B_1}|y|^{a}\D u\cdot\D \phi\, dxdy &= \int_{B_1^+} \int_{\mathbb{S}^{n-1}}r^{a+n+1}(\D_x {\tilde u},\partial_r \tilde u) 
    \cdot
    (\D_x \phi,\partial_r \phi )\, d\sigma dxdr \\
    &= \int_{B_1^+} r^{a+n+1}(\D_x {\tilde u},\partial_r {\tilde u}) 
    \cdot
    \Big(\int_{\mathbb{S}^{n-1}}(\D_x \phi,\partial_r \phi ) d\sigma\Big) dxdr \\
    &= \int_{B_1^+} r^{a+n-1} (\D_x {\tilde u},\partial_r {\tilde u}) 
    \cdot (\D_x \tilde \phi,\partial_r \tilde \phi ) dxdr = 0.
\end{align*}
    Therefore, $u$ is a weak solution to \eqref{eq:weak:sol:conormal:0}. The proof is complete.
\end{proof}

\begin{Theorem}[Smoothness of axially symmetric solutions]\label{axiallysmooth}
Let $a+n >0$ and let $u$ be axially symmetric in $y$ and $L_a$-harmonic in $B_1$ across $\Sigma_0$ (weak solution to \eqref{eq:weak:sol:conormal:0}). Then, $u \in C^\infty_{\rm loc}(B_{1})$.
\end{Theorem}

\begin{proof}

By definition $u(x,y)=\tilde{u}(x,|y|)$, for some function $\tilde{u}:B_1^+\subset \R^{d-n+1}\to \R$ and by using Lemma \ref{L:axially:1-1} one has that $\tilde{u}=\tilde{u}(x,r)$ is a weak solution to \eqref{eq:symm2}. By using the regularity theory of weighted equations degenerating on a hyperplane (see \cite{SirTerVit21a,TerTorVit24a}), noting that $a+n-1>-1$, we get that $\tilde{u} \in C^{\infty}_{\rm loc}(\overline{B_{r}^+})$, for every $r \in (0,1)$.

Next, by using \cite[Lemma 7.3]{SirTerVit21a}, the function $\tilde v(x,r):= r^{-1}\partial_r \tilde u(x,r)$ is a weak solution to 
\begin{equation*}
\begin{cases}
-\mathrm{div}(r^{a+n+1}\nabla \tilde v)=0, &\mathrm{in \ } B_{3/4}^+,\\
\displaystyle\lim_{r\to 0}r^{a+n+1}\partial_r \tilde v = 0, &  \mathrm{on \ } B_{3/4}\cap\{r=0\}.
\end{cases}
\end{equation*}
Applying again Lemma \ref{L:axially:1-1}, the function 
\[
{v}(x,y):= \tilde{v}(x,|y|)= \frac{\D u(x,y)\cdot y}{|y|^2},
\]
is a weak solution to 
\begin{equation}\label{eq:axially:derivative}
-\mathrm{div}(|y|^{a+2}\nabla {v})=0, \quad\mathrm{in \ } B_{3/4}.
\end{equation}
Next, we prove that $\D u \in L^\infty(B_{1/2})$. For every $j=1,\dots,d-n$, \cite[Proposition 5.1]{CorFioVit25} ensures that $\partial_{x_j} u$ is a weak solution to \eqref{eq:weak:sol:conormal:0}. By using \cite[Theorem 1.1]{CorFioVit25} we get that $\partial_{x_j} u \in C^{0,\alpha}(B_{1/2})$ and then $\partial_{x_j} u \in L^\infty(B_{1/2})$. On  the other hand, since $u$ is an axially symmetric in $y$ function, one has that
$$\D_y u(x,y) = \partial_r \tilde{u}(x,|y|)\frac{y}{|y|} \in L^\infty(B_{1/2}),$$
so $\D u \in L^\infty(B_{1/2})$. Hence, we have proved that $u \in C^{0,1}(B_{1/2})\subset H^1(B_{1/2})$ and we 
can rewrite equation \eqref{eq:weak:sol:conormal:0} as 
\begin{equation}\label{eq:axially:Delta:u}
    -\Delta u = a {v},\quad \text{ in } B_{1/2}.
\end{equation}
Since $v$ is a weak solution to \eqref{eq:axially:derivative}, by using again \cite[Theorem 1.1]{CorFioVit25} one has that $v \in C^{0,\alpha}(B_{1/2})$ for some $\alpha\in (0,1)$. Then, by applying classical Schauder regularity theory to \eqref{eq:axially:Delta:u}, we get that $u \in C^{2,\alpha}(B_{1/3})$.

Resuming, we have shown that if $u$ is an axially symmetric in $y$ weak solution to \eqref{eq:weak:sol:conormal:0}, then $u \in C^{2,\alpha}(B_{1/3})$. Therefore, since ${v}$ is an axially symmetric in $y$ weak solution to \eqref{eq:axially:derivative}, we get that ${v} \in C^{2,\alpha'}(B_{1/3})$, for some $\alpha' \in (0,1)$ and, by using again classical Schauder regularity theory to \eqref{eq:axially:Delta:u}, this implies $u \in C^{4,\alpha'}(B_{1/4})$. By iterating this procedure and using a covering lemma our statement follows.
\end{proof}

\section{Inhomogeneous conormal boundary conditions on the characteristic thin set}

In this section, we extend our regularity results by considering equations that satisfy a non homogeneous conormal condition on the thin set $ \Sigma_0 $. As discussed in Remark \ref{rem:conormal}, the weak formulation of equation \eqref{eq:weak:sol:conormal:0}, which holds across $ \Sigma_0 $, implies the fact that the solutions formally satisfy a homogeneous conormal condition on $\Sigma_0$. In the mid-range $ a+n \in (0,2) $, it is possible to impose inhomogeneous conormal boundary conditions on $ \Sigma_0 $.

\begin{Definition}\label{Def:inhomogeneous.conormal}
    Let $2\le n < d$ and $a+n \in (0,2)$. We say that $u$ is a weak solution to 
    \begin{equation}\label{eq:weak:non:homo}
    \begin{cases}
        -\dive(|y|^a \D u) = 0, & \text{ in } B_1,\\
        \displaystyle-\lim_{|y|\to 0}
        |y|^{a+n-1} \D u\cdot \frac{y}{|y|} = g(x), & \text{ on }\Sigma_0 \cap B_1,
    \end{cases}
    \end{equation}
    if $u \in H^{1,a}(B_1)$ and satisfies
    \[
    \int_{B_1}|y|^a \D u \cdot \D \phi dz = \omega_n\int_{\Sigma_0 \cap B_1} g(x) \phi(x,0) dx, \qquad \text{for every }\phi\in C_c^\infty(B_1),
    \]
    where $\omega_n = |\mathbb{S}^{n-1}|$. 
\end{Definition}

Before stating the main result of this section, we first discuss the expected regularity for the solutions of the equation. Specifically, the function $ |y|^{2-a-n} $ solves equation \eqref{eq:weak:non:homo} with $ g(x) = cost $, and this provides an upper bound on the regularity. In particular, when $ a + n \in (1,2) $, we expect H\"older continuity of solutions, and when $ a + n \in (0,1) $, we expect H\"older continuity of their gradient. Let us remark that $ |y|^{2-a-n} $ also solves the homogeneous Dirichlet problem \eqref{eq:Dirichlet:homogeneous} whenever $a+n<2$, see Section \ref{sec:9}.

Then, let us recall the exponent
\begin{equation}\label{alphastar2}
    \alpha_*=\alpha_*(a,n)= \frac{2-a-n + \sqrt{(2-a-n)^2 + 4(n-1)}}{2}.
\end{equation}
As established in \cite{CorFioVit25}, it represents a upper bound for the regularity degree of solutions having a homogeneous conormal boundary condition at $\Sigma_0$.

\begin{Proposition}\label{Prop:inhomogeneous}
    Let $2\le n < d$, $a+n \in (0,2)$ and $u$ be a weak solution to \eqref{eq:weak:non:homo}. Then, the following holds true.
    \begin{itemize}
        \item [i)] Let $g\in L^p(\Sigma_0\cap B_1)$, with $p>\frac{d-n}{2-a-n}$ and $\alpha \in (0,2-a-n-\frac{d-n}{p}]\cap (0,1)$. Then,  $u \in C^{0,\alpha}_{\rm loc}(B_1)$ and there exists a constant $C>0$ depending only on $d$, $n$, $a$, $p$ and $\alpha$ such that 
        \[
        \|u\|_{C^{0,\alpha}(B_{1/2})} \le C \big( 
        \|u\|_{L^{2,a}(B_{1})}
        +
        \|g\|_{L^{p}(\Sigma_0 \cap B_1)}
        \big).
        \]
        \item [ii)] Let us suppose that $a+n \in (0,1)$. Let $g\in L^p(\Sigma_0\cap B_1)$, with $p>\frac{d-n}{1-a-n}$ and $\alpha \in (0,1-a-n-\frac{d-n}{p}]$. Then,  $u \in C^{1,\alpha}_{\rm loc}(B_1)$ and there exists a constant $C>0$ depending only on $d$, $n$, $a$, $p$ and $\alpha$ such that 
        \[
        \|u\|_{C^{1,\alpha}(B_{1/2})} \le C \big( 
        \|u\|_{L^{2,a}(B_{1})}
        +
        \|g\|_{L^{p}(\Sigma_0 \cap B_1)}
        \big).
        \]
        \item [iii)] Let us suppose that $a+n \in (0,1)$ and let $\alpha \in (0,1-a-n)$. Let  $g\in C^{0,\alpha}(\Sigma_0\cap B_1)$. Then,  $u \in C^{1,\alpha}_{\rm loc}(B_1)$ and there exists a constant $C>0$ depending only on $d$, $n$, $a$ and $\alpha$ such that 
        \[
        \|u\|_{C^{1,\alpha}(B_{1/2})} \le C \big( 
        \|u\|_{L^{2,a}(B_{1})}
        +
        \|g\|_{C^{0,\alpha}(\Sigma_0\cap B_1)}
        \big).
        \]
    \end{itemize}

\end{Proposition}

\begin{proof}
We provide the proof of \emph{iii)} only, as the proofs of $i)$ and $ii)$ follow from the same arguments.
 
Let us consider the weak solution $\tilde{u}_1 = \tilde{u}_1(x,r) : B_1^+\subset \R^{d-n+1}\to \R$ to
    \begin{equation*}
    \begin{cases}
        -\dive(r^{a+n-1} \D \tilde{u}_1) = 0, & \text{ in } B_1^+,\\
        \displaystyle-\lim_{r\to 0}
        r^{a+n-1} \partial_r \tilde{u}_1 = g(x), & \text{ on }B_1\cap\{r=0\},\\
        \tilde{u}_1=0, & \text{ on } \partial B_1\cap\{r>0\}.
    \end{cases}
    \end{equation*}
By \cite[Theorem 8.4]{SirTerVit21a}, it follows that $\partial_r \tilde u_1 =0$ on $\Sigma_0\cap B_1$ and
\[
\|\tilde u_1\|_{C^{1,\alpha}(B_{1/2}^+)} \le C 
        \|g\|_{C^{0,\alpha}(\Sigma_0\cap B_1)},
\]
where $C>0$ depends only on $d$, $n$, $a$ and $\alpha$. Hence, by applying Lemma \ref{L:axially:1-1} with a minor modification, one has that $u_1(x,y):=\tilde u_1(x,|y|)$ is an axially symmetric in $y$ solution to 
    \begin{equation*}
    \begin{cases}
        -\dive(|y|^a \D u_1) = 0, & \text{ in } B_1,\\
        \displaystyle-\lim_{|y|\to 0}
        |y|^{a+n-1} \D u_1 \cdot \frac{y}{|y|} = g(x), & \text{ on }\Sigma_0 \cap B_1,\\
        u_1=0, &\text{ on }\partial B_1.
    \end{cases}
    \end{equation*}
Since $\partial_r \tilde{u}_1 = 0$ in $\Sigma_0 \cap B_1$, following computations analogous to those in \cite[Lemma 7.5]{SirTerVit21a} and \cite[Lemma 2.9]{AudFioVit24b}, we conclude that $u_1$ inherits the same regularity as $\tilde{u}_1$; that is,
\[
\| u_1\|_{C^{1,\alpha}(B_{1/2})} \le C 
        \|g\|_{C^{0,\alpha}(\Sigma_0\cap B_1)}.
\]
Next, let us define $u_2:= u-u_1$, which is a weak solution in the sense of Definition \ref{def:weak:sol:hom:0} to 
\[
-\dive(|y|^a \D u_2) =0,\quad \text{ in }B_1.
\]
Noticing that $\alpha_*=\alpha_*(n,a)$ given by \eqref{alphastar2} satisfies $\alpha_*-1>1-a-n > \alpha$, \cite[Theorem 1.2]{CorFioVit25} yields that
\[
\| u_2\|_{C^{1,\alpha}(B_{1/2})} \le C 
        \|u\|_{L^{2,a}(B_1)}.
\]
Hence, $u=u_1+u_2$ satisfies the desired estimate and the proof is complete.
\end{proof}

\section{Higher codimensional extensions of fractional Laplacians}\label{sec:10}

The aim of this section is to establish a connection between the degenerate equations discussed in this paper and fractional Laplacians on the very thin set $\Sigma_0$, in the sense of Dirichlet-to-Neumann maps, in the spirit of \cite{CafSil07}. Let $s\in(0,1)$ and $2\leq n<d$. The $s$-fractional Laplacian of sufficiently regular functions can be defined in $\R^{d-n}$ equivalently as a integro-differential operator
$$(-\Delta)^su(x)=C_{d-n,s}\lim_{\varepsilon\to0^+}\int_{\R^{d-n}\setminus B_\varepsilon(x)}\frac{u(x)-u(\xi)}{|x-\xi|^{d-n+2s}}\, d\xi\,,$$
or via Fourier transform 
$$\widehat{(-\Delta)^su}(\xi)=|\xi|^{2s}\hat u(\xi)\,.$$

Assuming $d-n >2s$, and using the fractional Hardy inequality, we can set fractional problems in $\mathcal D^s(\R^{d-n})$, defined as the completion of $C^\infty_c(\R^{d-n})$ with respect to the norm
$$\|u\|^2_{\mathcal D^s(\R^{d-n})}=\int_{\R^{d-n}} |(-\Delta)^{s/2}u(x)|^2\,.$$
Let $a=2-n-2s$, which satisfies $a+n\in(0,2)$. This corresponds to the mid-range regime, where the local weighted capacity of $\Sigma_0$ is positive and finite, see \cite[Lemma 2.4]{CorFioVit25}. Next, we define $\mathcal D^{1,a}(\R^d)$ as the completion of $C^\infty_c(\R^{d})$ with respect to
$$\|U\|^2_{\mathcal D^{1,a}(\R^{d})}=\int_{\R^d}|y|^a|\nabla U|^2,$$
which is a norm due to an Hardy-type inequality, see for instance \cite[Theorem 1]{CorMusNaz23}. In the  following result, we establish the existence of a trace operator from $\mathcal D^{1,a}(\R^d)$ to $\mathcal D^{s}(\R^{d-n})$. We refer also to \cite{Nek93} for some related local trace type results.
\begin{Lemma}\label{L:trace}
    Let $2\leq n<d$ and $a+n\in(0,2)$. Then, there exists a constant $C_{a,n}>0$ such that
    \begin{equation}\label{traceD1aDs}
        C_{a,n}\|\phi_{|\Sigma_0}\|_{\mathcal D^s(\R^{d-n})}\leq \|\phi\|_{\mathcal D^{1,a}(\R^{d})}
    \end{equation}
    for any $\phi\in C^\infty_c(\R^d)$ with $s=(2-a-n)/2\in(0,1)$. Moreover, if $d+a>2$ (i.e. $d-n>2s$), \eqref{traceD1aDs} defines the continuous trace operator
    \begin{equation*}
        \mathrm{Tr}: \mathcal D^{1,a}(\R^d)\to \mathcal D^{s}(\R^{d-n}).
    \end{equation*}
\end{Lemma}
\begin{proof}
The proof is an adaptation of \cite[Theorem 3.8]{CorMus22}, so we omit most of the details. Let $\phi \in C^\infty_c(\R^d)$, and let $\hat \phi$ denote the Fourier transform of $\phi$ with respect to the $x$ variables. Then we have 
\[
\int_{\R^d} |y|^a |\nabla \phi|^2 dz = \int_{\R^{d-n}}\int_{\R^n} |y|^a(|\xi|^2|\hat \phi|^2 + |\nabla_y \hat \phi|^2)dy d\xi\,.
\]
Next, define $v(\xi, t) = \hat\phi(\xi, |\xi|^{-1} t)$, and perform the change of variables $y = |\xi|^{-1} t$ in the integral above. This leads to 
\[
\int_{\R^d} |y|^a |\nabla \phi|^2 dz = \int_{\R^{d-n}}|\xi|^{2-a-n}\int_{\R^n} |t|^a(|v|^2 + |\nabla_t v|^2)dt\, d\xi\,.
\]
We now claim that 
\[
\inf_{v \in C^\infty_c(\R^n), v(0) = 1} \int_{\R^n} |t|^a(|v|^2 + |\nabla_t v|^2)dt = C_{a,n} >0\,.
\]
Assume that this claim is true. Then, recalling that $v(\xi, 0) = \hat \phi(\xi, 0)$, we obtain 
\[
\int_{\R^d} |y|^a |\nabla \phi|^2 dz \geq C_{a,n} \int_{\R^{d-n}}|\xi|^{2-a-n}|\hat \phi(\xi, 0)|^2 d\xi = C_{a,n} \int_{\R^{d-n}} |(-\Delta)^{s/2}\phi(x, 0)|^2 dx \,,
\]
with $s=(2-a-n)/2\in(0,1)$. This completes the proof of \eqref{traceD1aDs}. The rest of the proof follows in a standard way. 

Let us now prove the claim. To do this, we argue by contradiction. Assume that $C_{a,n} = 0$. Then, for every $\delta >0$ there exists $v_\delta \in C^\infty_c(\R^n)$ such that $v_\delta(0) =1$ and 
\begin{equation}\label{eq:tracpass}
\int_{\R^n} |t|^a(|v|^2 + |\nabla_t v|^2)dt < \delta\,.
\end{equation}
Using computations similar to those in the proof of part $iii)$ of \cite[Lemma 2.4]{CorFioVit25}, we find that for every $\rho >0$,
\[
n \omega_n = \int_{\mathbb{S}^{n-1}}|v_\delta(0)|^2 d \sigma \leq c \rho^{2-a-n} \int_{\R^{n}}|t|^a |\nabla_t v_\delta|^2 dt + 2 \int_{\mathbb{S}^{n-1}}|v_\delta(\rho \sigma )|^2 d\sigma\,.
\]
Multiplying by $\rho^{a+n -1}$, integrating over $\rho \in (0,1)$, and using \eqref{eq:tracpass}, we obtain
\[
\frac{n\omega_n}{a+n} \leq c \int_{\R^{n}}|t|^a |\nabla_t v_\delta|^2 dt + 2 \int_{B_1}|t|^a|v_\delta|^2 dt \leq c\delta\,,
\]
which is a contradiction for sufficiently small $\delta$. Therefore, the claim holds true and the proof is complete. 
\end{proof}
As a dual result, we define a unique minimal energy extension in $\mathcal D^{1,a}(\R^{d})$ for functions in $\mathcal D^{s}(\R^{d-n})$.
\begin{Lemma}
   Let $2\leq n<d$ with $d+a>2$ and $a+n\in(0,2)$. Let $s=(2-a-n)/2$. Then, for any $u\in\mathcal D^{s}(\R^{d-n})$ the minimization problem
   \begin{equation}\label{extensionDsD1a}
       \inf_{U\in \mathcal D^{1,a}(\R^d) \,, \, \mathrm{Tr}U=u}\|U\|^2_{\mathcal D^{1,a}(\R^d)}
   \end{equation}
   admits a unique minimizer. Moreover, \eqref{extensionDsD1a} defines an extension operator
   \begin{equation*}
        \mathrm{Ext}: \mathcal D^{s}(\R^{d-n})\to \mathcal D^{1,a}(\R^d).
    \end{equation*}
\end{Lemma}
\begin{proof}
Fix $u \in \mathcal D^s(\R^{d-n})$. The existence of a minimizer $\bar U$ for \eqref{extensionDsD1a} follows from standard variational arguments. In particular, $\bar U$ satisfies
\begin{equation}\label{eq:mineq}
\int_{\R^d}|y|^a \nabla \bar U \cdot \nabla \phi \, dz = 0\,, \qquad \text{ for every } \phi \in \mathcal D^{1,a}(\R^d)\,, \ \mathrm{Tr} \phi = 0\,.
\end{equation}
Finally, to prove the uniqueness of $\bar U$, it is sufficient to use a standard contradiction argument to show that \eqref{eq:mineq} admits an unique solution in the set $\{ U \in \mathcal D^{1,a}(\R^n) \mid \mathrm{Tr}U = u\}$. 
\end{proof}
Then, the extension result can be summarized as below.
\begin{Proposition}\label{P:extension}
    Let $2\leq n<d$ with $d+a>2$ and $a+n\in(0,2)$. Let $s=(2-a-n)/2\in(0,1)$ Then for any $u\in\mathcal D^{s}(\R^{d-n})$, the extension $U=\mathrm{Ext}(u)$ is given by
    \begin{equation*}
        U(x,y)=u\ast P(x,y),\qquad P(x,y)=\frac{\Gamma(\frac{d-n + 2s}{2})}{\pi^{\frac{d-n}{2}}\Gamma(s)}\frac{|y|^{2s}}{(|x|^2+|y|^2)^{\frac{d-n+2s}{2}}}\,,
    \end{equation*}
    and is solution to
$$\begin{cases}
-\mathrm{div}(|y|^a\nabla U)=0 &\mathrm{in \ }\R^d \setminus \Sigma_0\\ 
\displaystyle-\lim_{|y|\to0}|y|^{a+n-1}\nabla U\cdot\frac{y}{|y|}=d_{a,n}(-\Delta)^su &\mathrm{on \ } \Sigma_0\,,
\end{cases}$$
in the sense that $U \in \mathcal{D}^{1,a}(\R^d)$ and satisfies
\[
\int_{\R^{d}} |y|^a \nabla  U \cdot \nabla \phi\, dx\, dy = d_{a,n}\int_{\R^{d-n}}(-\Delta)^s u \, \phi(x, 0) \,dx\,, \quad \text{ for every } \phi\in C^{\infty}_c(\R^{d})\,,
\]
where $d_{a,n} = 2^{a+n-1}\Gamma(\frac{a+n}{2})/\Gamma(\frac{2-a-n}{2})$.
\end{Proposition}

\begin{proof}
Let $u \in \mathcal D^s(\R^n)$. Notice that $P(x,y)=P(|x|,|y|)$, and hence $U(x,y)=\tilde U(x,|y|)$, where $\tilde U(x, r) = u * P(x, r)$ for $r>0$. From the codimension $1$ extension theory in \cite{CafSil07}, we have that $\tilde U \in \mathcal D^{1, 1-2s}(\R^{d-n})$, and 
\[
 \begin{cases}
-\mathrm{div}(r^{1-2s}\nabla \tilde U)=0 &\mathrm{in \ }\R^{d-n+1}_+ \\
\displaystyle-\lim_{r\to0}r^{1-2s}\partial_r\tilde U= 2^{1-2s}\frac{\Gamma(1-s)}{\Gamma(s)}(-\Delta)^su &\mathrm{on \ } \R^{d-n}\,,
\end{cases}
\]
or, equivalently, that 
\[
\int_{\R^{d-n+1}_+} r^{a+ n -1} \nabla \tilde U \cdot \nabla \tilde \phi\, dx\, dr = d_{a,n}\int_{\R^{d-n}}(-\Delta)^s u \,\tilde \phi(x, 0) \,dx\,, \quad \text{ for every } \tilde \phi\in C^{\infty}_c(\overline{\R^{d-n+1}_+})\,,
\]
where we used that $s=(2-a-n)/2$.
Thus, adapting the argument of the proof of Lemma \ref{L:axially:1-1} we infer that $U \in \mathcal D^{1,a}(\R^d)$, $\mathrm{Tr} U = u$ and 
\[
\int_{\R^{d}} |y|^a \nabla  U \cdot \nabla \phi\, dx\, dy = d_{a,n}\int_{\R^{d-n}}(-\Delta)^s u \, \phi(x, 0) \,dx\,, \quad \text{ for every } \phi\in C^{\infty}_c(\R^{d})\,. 
\]
In particular, $U$ is the unique solution of \eqref{eq:mineq} in the set $\{ U \in \mathcal D^{1,a}(\R^n) \mid \mathrm{Tr}U = u\}$, and thus, the unique minimizer of \eqref{extensionDsD1a}. It immediately follows that $U = \mathrm{Ext}(u)$, and the proof is complete. 
\end{proof}

\begin{remark}
Let $u \in \mathcal D^s(\R^{d-n})$. As shown in Proposition \ref{P:extension}, it holds
\[
\int_{\R^{d}} |y|^a \nabla \mathrm{Ext}(u) \cdot \nabla \phi\, dx\, dy = d_{a,n}\int_{\R^{d-n}}(-\Delta)^s u \, \phi(x, 0) \,dx\,, \quad \text{ for every } \phi \in \mathcal{D}^{1,a}(\R^d)\,. 
\]
By taking $\phi = \mathrm{Ext}(u)$, we immediately infer $\|\mathrm{Ext}(u)\|_{\mathcal D^{1,a}(\R^d)} = d_{a,n}\|u\|_{\mathcal D^s(\R^{d-n})}$. As a consequence, $C_{a,n} = d_{a,n}$, where $C_{a,n}$ is the constant in Lemma \ref{L:trace}. 
\end{remark}

\section{Homogeneous Dirichlet problem via a boundary Harnack principle}\label{sec:9}

In this section, we prove H\"older $C^{0,\alpha}$ and Schauder $C^{1,\alpha}$ regularity for solutions to the homogeneous Dirichlet problem
\begin{equation}\label{eq:Dirichlet:homogeneous}
\begin{cases}
-\mathrm{div}(|y|^a\nabla u)=0 &\mathrm{in \ } B_1\setminus\Sigma_0\\
u=0 &\mathrm{on \ } B_1\cap\Sigma_0\,,
\end{cases}
\end{equation}
whenever $a+n<2$. The solutions we are referring to are elements of $\tilde H^{1,a}(B_1)$, which is defined as the completion of $C_c^\infty(\overline{B_1}\setminus\Sigma_0)$ with respect to the norm
$\|\cdot\|_{H^{1,a}(B_1)}$ and satisfies
\begin{equation*}
\int_{B_1}|y|^a\nabla u\cdot\nabla\phi\,dz=0\,, \qquad \text{for every } \phi\in C^{\infty}_c(B_1\setminus\Sigma_0).
\end{equation*}
We remark that in the supersingular case $a+n\le 0$, any function in $H^{1,a}(B_1)$ is forced to have trace $u=0$ on $\Sigma_0$ (see \cite[Proposition 2.5]{CorFioVit25}). In the mid-range $a+n \in (0,2)$, the trace condition $u=0$ on $\Sigma_0$ is also well-defined, since $\Sigma_0$ has positive local weighted capacity. See also \cite{Nek93} for some trace type theorems on lower dimensional sets.

The idea is to obtain regularity as a corollary of \cite[Theorems 1.1 and 1.2]{CorFioVit25} via a boundary Harnack principle, in the same spirit as \cite{SirTerVit21b}. In fact, the ratio $w:=u/u_0$ between a solution $u$ to \eqref{eq:Dirichlet:homogeneous} and the \emph{characteristic solution} to \eqref{eq:Dirichlet:homogeneous}
\begin{equation*}
u_0(y)=|y|^{2-a-n}
\end{equation*}
solves the equation
\begin{equation}\label{eq:ratio:homogeneous}
-\mathrm{div}(|y|^b\nabla w)=0 \qquad\mathrm{in \ } B_1 
\end{equation}
with exponent $b=4-a-2n$ lying in the superdegenerate range since $b+n=4-a-n>2$.

\begin{Lemma}
Let $a+n<2$, $u$ be a solution to \eqref{eq:Dirichlet:homogeneous} and $u_0=|y|^{2-a-n}$. Then, $w=u/u_0$ solves \eqref{eq:ratio:homogeneous}.
\end{Lemma}
\begin{proof}
The equation is trivially satisfied far from $\Sigma_0$ in a classic sense. Moreover, due to the superdegeneracy of the weight $|y|^b$, and combining \cite[Lemma 2.8]{CorFioVit25} with the density of $C^\infty_c(\overline {B_1}\setminus\Sigma_0)$ in $H^{1,b}(B_1)$, the result is trivially true if one shows that the $H^{1,b}(B_1)$-energy of $w$ is finite. Recalling the Hardy inequality (see \cite[Proposition 2.2]{CorFioVit25}), we easily compute that
    \begin{align*}
        &\int_{B_1}|y|^b |\D w|^2 \,dz\le c \int_{B_1}|y|^b \Big(\frac{|\D u|^2}{|u_0|^2} +
        \frac{u^2 |\D u_0|^2}{|u_0|^4} 
        \Big)\,dz\, \\
        &= c \Big(\int_{B_1}|y|^a |\D u|^2 \,dz +  \int_{B_1}|y|^{a-2}  u^2 \,dz
        \Big)
        \le c \int_{B_1}|y|^a |\D u|^2 \,dz\,.
    \end{align*}  
    Furthermore, $\|w\|_{L^{2,b}(B_1)} =\|u\|_{L^{2,a}(B_1)} $, hence $w \in H^{1,b}(B_1)$ and the proof is complete.    
    \end{proof}

Let us recall that the homogeneity degree appearing \eqref{alphastar2}, related to the new exponent $b$, is given by  
\begin{equation}\label{expblioville}
    \alpha_*(b,n) = \frac{2-b-n + \sqrt{(2-b-n)^2 + 4(n-1) }}{2}\,.
\end{equation}
Then, \cite[Theorems 1.1 and 1.2]{CorFioVit25} imply the following result.

\begin{Corollary}\label{C:Dirichlet}
    Let $a+n<2$ and $b=4-a-2n$. Let $\alpha_*(b,n)>0$ be the homogeneity exponent in \eqref{expblioville} and 
\begin{equation*}
       \alpha\in (0,1)\cap(0,\alpha_*(b,n)).
\end{equation*}
Let $u$ be a weak solution to \eqref{eq:Dirichlet:homogeneous} in $B_1$ and $u_0=|y|^{2-a-n}$.
Then, $u/u_0\in C^{0,\alpha}_{\rm loc}(B_{1})$ and there exists a constant $c>0$ depending only on $d$, $n$, $a$ and $\alpha$ such that
\begin{equation*}
\Big\|\frac{u}{u_0}\Big\|_{C^{0,\alpha}(B_{1/2})}\le c
\|u\|_{L^{2,a}(B_{1})}.
\end{equation*}
If moreover $\alpha_*(b,n)>1$ and
\begin{equation*}
       \alpha\in (0,1)\cap(0,\alpha_*(b,n)-1),
\end{equation*}
then $u/u_0\in C^{1,\alpha}_{\rm loc}(B_{1})$ and there exists a constant $c>0$ depending only on $d$, $n$, $a$ and $\alpha$ such that
\begin{equation*}
\Big\|\frac{u}{u_0}\Big\|_{C^{1,\alpha}(B_{1/2})}\le c
\|u\|_{L^{2,a}(B_{1})}.
\end{equation*}
\end{Corollary}

\begin{remark}\label{rem:DFM}
The previous result improves \cite[Theorem 1]{Ngu11} in the flat case, which corresponds to a $L^\infty$-bound of $u/u_0$. Moreover, it slightly improve \cite[Theorems 1.2 and 1.3]{Fio24} in the present homogeneous case. The exponent $\alpha_*(b,n)$, given by \eqref{expblioville}, satisfies $\alpha_*(b,n)>2-a-n$ whenever
    \begin{equation}\label{codimcond}
        n-1>2(2-a-n)^2.
    \end{equation}
In particular this implies that, under \eqref{codimcond}, one can imply the sharp $C^{2-a-n}$ regularity of solutions to \eqref{eq:Dirichlet:homogeneous}. In fact, if $w=u/u_0$ has $C^{\beta}$ regularity with $\beta\geq 2-a-n$, then
\[
u=wu_0\in C^{2-a-n}.
\]
In particular, when $a+n=1$, which corresponds to the exponent studied by David, Feneuil and Mayboroda (see \cite{DavFenMay19} and its related works), one can prove the sharp Lipschitz continuity of solutions whenever the codimension $n\geq4$.

\end{remark}

\section*{Acknowledgment}
The authors are research fellows of Istituto Nazionale di Alta Matematica INDAM group GNAMPA and supported by the GNAMPA project E5324001950001 \emph{PDE ellittiche che degenerano su variet\`a di dimensione bassa e frontiere libere molto sottili}. G.C. is supported by the PRIN project 20227HX33Z \emph{Pattern formation in nonlinear phenomena}. S.V. is supported by the PRIN project 2022R537CS \emph{$NO^3$ - Nodal Optimization, NOnlinear elliptic equations, NOnlocal geometric problems, with a focus on regularity}.

\end{document}